\documentclass[11pt,leqno]{article}
\usepackage[margin=1in]{geometry} 
\usepackage{amssymb,amsfonts,amsmath,bbm,mathrsfs,stmaryrd,mathtools}
\usepackage{xcolor}
\usepackage{url}

\usepackage{graphicx}

\usepackage{accents}

\usepackage{extarrows}

\usepackage[shortlabels]{enumitem}
\usepackage{tensor}

\usepackage{xr}
\usepackage[T1]{fontenc}
\usepackage[utf8]{inputenc}
\usepackage[colorlinks,
linkcolor=black!75!red,
citecolor=blue,
pdftitle={},
pdfproducer={pdfLaTeX},
pdfpagemode=None,
bookmarksopen=true,
bookmarksnumbered=true,
backref=page]{hyperref}

\usepackage{tikz}
\usetikzlibrary{arrows,calc,decorations.pathreplacing,decorations.markings,decorations.shapes,intersections,shapes.geometric,through,fit,shapes.symbols,positioning,decorations.pathmorphing}

\makeatletter
\newlength\zig@L
\newlength\zig@La
\newlength\zig@Lb

\newcommand{\xzigrightarrow}[2][]{%
  \mathrel{%
    \settowidth{\zig@La}{$\scriptstyle #2$}%
    \settowidth{\zig@Lb}{$\scriptstyle #1$}%
    \zig@L=\zig@La\relax
    \ifdim\zig@Lb>\zig@L \zig@L=\zig@Lb\fi
    \advance\zig@L by 2.2em\relax
    \tikz[baseline=-0.65ex]{%
      \draw[->,
            line cap=round,
            decorate,
            decoration={zigzag,segment length=4pt,amplitude=1.1pt}]%
        (0,0) -- (\zig@L,0)
        node[midway,above=2pt] {$\scriptstyle #2$}%
        \if\relax\detokenize{#1}\relax\else
          node[midway,below=2pt] {$\scriptstyle #1$}%
        \fi
      ;
    }%
  }%
}
\makeatother

\makeatletter
\newcommand{\squigjoin}{1mu} % tune this: 0.5mu, 1mu, 1.5mu, ...

\def\sqleft@{\sim}                    % no overlap here
\def\sqmid@{\sim\mkern-\squigjoin}    % overlap only between repeated mids

\def\rightsquigarrowfill@{%
  \arrowfill@{\sqleft@}{\sqmid@}{\mkern-4mu\succ}%
}

\newcommand{\xrightsquigarrow}[2][]{%
  \ext@arrow 0359\rightsquigarrowfill@{#1}{#2}%
}
\makeatother

\makeatletter
\newcommand*\circled[1]{\tikz[baseline=(char.base)]{
    \node[shape=circle, draw, inner sep=0pt, 
    minimum height={\f@size},] (char) {\vphantom{WAH1g}#1};}}
\makeatother

\makeatletter
\DeclareRobustCommand\widecheck[1]{{\mathpalette\@widecheck{#1}}}
\def\@widecheck#1#2{%
    \setbox\z@\hbox{\m@th$#1#2$}%
    \setbox\tw@\hbox{\m@th$#1%
       \widehat{%
          \vrule\@width\z@\@height\ht\z@
          \vrule\@height\z@\@width\wd\z@}$}%
    \dp\tw@-\ht\z@
    \@tempdima\ht\z@ \advance\@tempdima2\ht\tw@ \divide\@tempdima\thr@@
    \setbox\tw@\hbox{%
       \raise\@tempdima\hbox{\scalebox{1}[-1]{\lower\@tempdima\box
\tw@}}}%
    {\ooalign{\box\tw@ \cr \box\z@}}}
\makeatother

\usepackage{braket}

\usepackage[amsmath,thmmarks,hyperref]{ntheorem}
\usepackage{cleveref}

\newcommand\nthalias[1]{\AddToHook{env/#1/begin}{\crefalias{lemma}{#1}}}

\nthalias{definition}
\nthalias{example}
\nthalias{examples}
\nthalias{remark}
\nthalias{remarks}
\nthalias{convention}
\nthalias{notation}
\nthalias{construction}
\nthalias{sketch}
\nthalias{theoremN}
\nthalias{propositionN}
\nthalias{corollaryN}
\nthalias{lemma}
\nthalias{proposition}
\nthalias{corollary}
\nthalias{theorem}
\nthalias{conjecture}
\nthalias{question}
\nthalias{assumption}

\creflabelformat{enumi}{#2#1#3}

\crefname{section}{Section}{Sections}
\crefformat{section}{#2Section~#1#3} 
\Crefformat{section}{#2Section~#1#3} 

\crefname{subsection}{\S}{\S\S}
\AtBeginDocument{%
  \crefformat{subsection}{#2\S#1#3}%
  \Crefformat{subsection}{#2\S#1#3}% 
}

\crefname{subsubsection}{\S}{\S\S}
\AtBeginDocument{%
  \crefformat{subsubsection}{#2\S#1#3}%
  \Crefformat{subsubsection}{#2\S#1#3}% 
}

\theoremstyle{plain}

\newtheorem{lemma}{Lemma}[section]
\newtheorem{proposition}[lemma]{Proposition}
\newtheorem{corollary}[lemma]{Corollary}
\newtheorem{theorem}[lemma]{Theorem}

\theoremstyle{plain}
\theoremnumbering{Alph}

\theoremstyle{plain}
\theorembodyfont{\upshape}
\theoremsymbol{\ensuremath{\blacklozenge}}

\newtheorem{definition}[lemma]{Definition}
\newtheorem{example}[lemma]{Example}

\newtheorem{remark}[lemma]{Remark}

\newtheorem{notation}[lemma]{Notation}

\crefname{definition}{definition}{definitions}
\crefformat{definition}{#2definition~#1#3} 
\Crefformat{definition}{#2Definition~#1#3} 

\crefname{ex}{example}{examples}
\crefformat{example}{#2example~#1#3} 
\Crefformat{example}{#2Example~#1#3} 

\crefname{exs}{example}{examples}
\crefformat{examples}{#2example~#1#3} 
\Crefformat{examples}{#2Example~#1#3} 

\crefname{remark}{remark}{remarks}
\crefformat{remark}{#2remark~#1#3} 
\Crefformat{remark}{#2Remark~#1#3} 

\crefname{remarks}{remark}{remarks}
\crefformat{remarks}{#2remark~#1#3} 
\Crefformat{remarks}{#2Remark~#1#3} 

\crefname{convention}{convention}{conventions}
\crefformat{convention}{#2convention~#1#3} 
\Crefformat{convention}{#2Convention~#1#3} 

\crefname{notation}{notation}{notations}
\crefformat{notation}{#2notation~#1#3} 
\Crefformat{notation}{#2Notation~#1#3} 

\crefname{table}{table}{tables}
\crefformat{table}{#2table~#1#3} 
\Crefformat{table}{#2Table~#1#3}

\crefname{lemma}{lemma}{lemmas}
\crefformat{lemma}{#2lemma~#1#3} 
\Crefformat{lemma}{#2Lemma~#1#3} 

\crefname{proposition}{proposition}{propositions}
\crefformat{proposition}{#2proposition~#1#3} 
\Crefformat{proposition}{#2Proposition~#1#3} 

\crefname{propositionN}{proposition}{propositions}
\crefformat{propositionN}{#2proposition~#1#3} 
\Crefformat{propositionN}{#2Proposition~#1#3} 

\crefname{corollary}{corollary}{corollaries}
\crefformat{corollary}{#2corollary~#1#3} 
\Crefformat{corollary}{#2Corollary~#1#3} 

\crefname{corollaryN}{corollary}{corollaries}
\crefformat{corollaryN}{#2corollary~#1#3} 
\Crefformat{corollaryN}{#2Corollary~#1#3} 

\crefname{theorem}{theorem}{theorems}
\crefformat{theorem}{#2theorem~#1#3} 
\Crefformat{theorem}{#2Theorem~#1#3} 

\crefname{theoremN}{theorem}{theorems}
\crefformat{theoremN}{#2theorem~#1#3} 
\Crefformat{theoremN}{#2Theorem~#1#3} 

\crefname{enumi}{}{}
\crefformat{enumi}{#2#1#3}
\Crefformat{enumi}{#2#1#3}

\crefname{assumption}{assumption}{Assumptions}
\crefformat{assumption}{#2assumption~#1#3} 
\Crefformat{assumption}{#2Assumption~#1#3} 

\crefname{construction}{construction}{Constructions}
\crefformat{construction}{#2construction~#1#3} 
\Crefformat{construction}{#2Construction~#1#3} 

\crefname{sketch}{sketch}{Sketches}
\crefformat{sketch}{#2sketch~#1#3} 
\Crefformat{sketch}{#2Sketch~#1#3} 

\crefname{question}{question}{Questions}
\crefformat{question}{#2question~#1#3} 
\Crefformat{question}{#2Question~#1#3} 

\crefname{equation}{}{}
\crefformat{equation}{(#2#1#3)} 
\Crefformat{equation}{(#2#1#3)}

\numberwithin{equation}{section}

\theoremstyle{nonumberplain}
\theoremsymbol{\ensuremath{\blacksquare}}

\newtheorem{proof}{Proof}

\newcommand\bZ{{\mathbb Z}}

\newcommand\cA{{\mathcal A}}

\newcommand\cC{{\mathcal C}}
\newcommand\cD{{\mathcal D}}

\newcommand\cS{{\mathcal S}}

\newcommand\ol{\overline}

\DeclareMathOperator{\id}{id}

\newcommand{\cat}[1]{\textsc{#1}}

\newcommand\spr[1]{\cite[\href{https://stacks.math.columbia.edu/tag/#1}{Tag {#1}}]{stacks-project}}

\newcommand{\comment}[1]{}

\newcommand{\xrightarrowdbl}[2][]{%
  \xrightarrow[#1]{#2}\mathrel{\mkern-14mu}\rightarrow
}

\title{Injectivity paucity in AB5 categories of oversize chains}
\author{Alexandru Chirvasitu}

\begin{document}

\date{}

\newcommand{\Addresses}{{% additional braces for segregating \footnotesize
  \bigskip
  \footnotesize

  \textsc{Department of Mathematics, University at Buffalo}
  \par\nopagebreak
  \textsc{Buffalo, NY 14260-2900, USA}  
  \par\nopagebreak
  \textit{E-mail address}: \texttt{achirvas@buffalo.edu}

}}

\maketitle

\begin{abstract}
  We construct examples of abelian categories with no non-zero injective (or projective) objects satisfying Grothendieck's AB5 condition. The procedure combines Rickard's examples of AB5 categories without products but some non-trivial injectives (also addressing an apparent gap in the literature) with a 2-functorial construct attaching to any category $\mathcal{C}$ that of $\mathcal{C}$-objects equipped with set-indexed families of endomorphisms. 
\end{abstract}


\vspace{.5cm}

\noindent{MSC 2020: 18E10; 18G05; 18A35; 18A30; 18A20; 18N10; 18E40; 03E10
  % 18E10   	Abelian categories, Grothendieck categories
  % 18G05   	Projectives and injectives (category-theoretic aspects)
  % 18A35   	Categories admitting limits (complete categories), functors preserving limits, completions  
  % 18A30   	Limits and colimits (products, sums, directed limits, pushouts, fiber products, equalizers, kernels, ends and coends, etc.)
  % 18A20   	Epimorphisms, monomorphisms, special classes of morphisms, null morphisms
  %% 18E35   	Localization of categories, calculus of fractions
  % 18N10  	2-categories, bicategories, double categories
  % 18E40   	Torsion theories, radicals
  % 03E10   	Ordinal and cardinal numbers
  
}

%\tableofcontents

%%%%%%%%%%%%%%%%%%%%%%%%%%%%%%%%
%%%%%%%%%%%%%%%%%%%%%%%%%%%%%%%%
\section*{Introduction}

Recall \cite[Theorem 2.4]{MR4092830} Rickard's example of an abelian category satisfying Grothendieck's \emph{AB5 condition} (\cite[Theorem 2.8.6 and subsequent discussion]{pop_ab-cat}: abelian, cocomplete, with exact filtered colimits), but which is nevertheless not complete (i.e. does not have arbitrary limits, because it does not have countable products; in AB-condition language, it is not \emph{AB3$^*$}):
\begin{itemize}[wide]
\item Fix a functor
  \begin{equation}\label{eq:kappa}
    \text{large poset of ordinal numbers}
    =:
    \left(\cat{ON},\le\right)
    \xrightarrow[\quad\alpha\mapsto \Bbbk_{\alpha}\quad]{\quad{\Bbbk}\quad}
    \cat{Fields}
  \end{equation}
  with infinite-degree $\Bbbk_{\alpha}<\Bbbk_{\beta}$ whenever $\alpha\lneq \beta$.

\item The attached category $\cC_{{\Bbbk}}$ then consists of (large) tuples $\left(V_{\alpha}\right)_{\alpha\in \cat{ON}}$ of $\Bbbk_{\alpha}$-linear spaces and maps
  \begin{equation*}
    V_{\alpha}
    \xrightarrow[\quad\alpha\le \beta\quad]{\quad\nu_{\beta\alpha}\quad}
    V_{\beta}
  \end{equation*}
  compatible with the $\Bbbk_{\alpha}$-linear structures in the obvious sense, with $(V_{\alpha},\nu_{\beta\alpha})$ \emph{$\alpha$-grounded} \cite[Definition 2.2]{MR4092830} for some $\alpha\in \cat{ON}$:
  \begin{equation*}
    \forall\left(\beta>\alpha\right)
    \left(V_{\beta}=\Bbbk_{\beta}\nu_{\beta\alpha}\left(V_{\alpha}\right)\right).
  \end{equation*}

\item There is an obvious embedding $\cC_{\Bbbk}\subset \ol{\cC}_{\Bbbk}$, the latter consisting of arbitrary $(V_{\alpha})$. It is not \emph{locally small}\footnote{A word of caution is in order concerning the term: we reference \cite{pop_ab-cat}, where ``locally small'' has a different meaning altogether; there, it means \emph{well-powered} (i.e. objects have only sets of subobjects rather than proper classes).}, in the sense \cite[\S V.8, post Corollary]{mcl_2e} that its hom-sets are generally proper classes. Before it becomes clear that the objects involved are grounded, abelian-category notions such as (co)kernels are to be understood as housed by $\ol{\cC}_{\Bbbk}$. 
\end{itemize}

What ultimately powers the example is the absence of a \emph{generator} \cite[Definition 4.5.1]{brcx_hndbk-1}: \emph{Grothendieck categories} (i.e. AB5 with a generator\footnote{Concerning references to \cite{freyd_abcats}, which do feature below, note the terminological discrepancy from what seems to be modern standard practice: in Freyd's book ``Grothendieck''=``AB5''.}: \cite[\S 2.8, Note 3]{pop_ab-cat}) are automatically AB3$^*$ by \cite[Corollary 3.7.10]{pop_ab-cat}. The motivation for the present note has been twofold:
\begin{itemize}[wide]
\item Given the essential role played by ``categorical-smallness'' constraints such as generators or set-sized families of subobjects in proving the existence of sufficiently many injectives (e.g. \cite[Lemmas 3.10.5/3.10.6 and Theorem 3.10.10]{pop_ab-cat}), it is not unnatural to examine the extent to which the failure of precisely such constraints in $\cC_{\Bbbk}$ affects its wealth of injectives.

\item In the process, an examination of the aforementioned \cite[Theorem 2.4]{MR4092830}'s proof seems to indicate that its claim of abelianness for the individual subcategories of $\alpha$-grounded objects cannot be valid: they are full, reflective, (co)complete and \emph{pre-abelian} \cite[\S 2.2]{pop_ab-cat}, but the \emph{$\alpha$-grounded} kernel of an $\alpha$-grounded-object morphism differs, generally, from its ``naive'' kernel: \Cref{le:trnc.chn.ess}, \Cref{cor:trnc.not.ab} and \Cref{ex:ker.grnd} expand. The abelianness of the target $\cC_{\Bbbk}$ does nevertheless appear to hold (and \cite[Theorem 2.4]{MR4092830} with it), as verified in \Cref{th:is.ab.no.env}.
\end{itemize}

A summary of the results, then, would be as follows:
\begin{enumerate}[(1),wide]
\item $\cC_{\Bbbk}$ is indeed AB5 (and incomplete), but does not have enough injectives (\Cref{th:is.ab.no.env}).

\item There \emph{are}, however, non-zero injectives: they are precisely those of the individual full, (co)reflective Grothendieck subcategories consisting respectively of objects $(V_{\alpha})_{\alpha}$ with $\alpha\ge \alpha_0\Rightarrow V_{\alpha}=\{0\}$ (one for each ordinal $\alpha$; \Cref{th:inj.event.vnsh}).

\item $\cC_{\Bbbk}$ embeds fully into an AB5 incomplete category $\cC_{\Bbbk}\triangleleft\cat{Set}$ with no non-zero injectives or projectives, consisting of $\cC_{\Bbbk}$-objects $c$ equipped with maps $S\to \cC_{\Bbbk}(c,c)$ for sets $S$ (\Cref{th:incmpl.no.inj}).

  It is this last gadget that can be extrapolated into what the abstract alludes to as a ``2-categorical construct'': $\bullet\triangleleft \cat{Set}$ accepts categories with a zero object as its inputs, and preserves a number of pleasant properties such as abelianness, AB conditions, etc. 
\end{enumerate}

% %

%%%%%%%%%%%%%%%%%%%%%%%%%%%%%%%%
%%%%%%%%%%%%%%%%%%%%%%%%%%%%%%%%
\section{AB5 categories of long morphism chains}\label{se:ab5.lng}

In the sequel, the set-theoretic formalism in place should be taken to be
\begin{itemize}[wide]
\item either the familiar ZFC (Zermelo-Fraenkel+Choice) of \cite[Introduction, \S 7]{kun_st_2nd_1983}, in which case proper classes have no formal standing and formulas involving them should be regarded as the ``circumlocutions in the metatheory'' referred to in \cite[\S I.12]{kun_st_2nd_1983};

\item or the NBG (von Neumann-Bernays-G\"odel) \emph{conservative extension} of ZFC referenced briefly in the same \cite[\S I.12]{kun_st_2nd_1983} (choice-enhanced: see the discussion in \cite[\S II.7.3]{fbhl_set_1973}): proper classes are then legitimate formal objects and the formulas below involving any such follow the NBG injunction that all \emph{bound} variables be sets.
\end{itemize}

\begin{remark}\label{re:no.gen.no.wp}
  The lack of a generator manifests in other ways: abelian categories admitting a generator are \emph{well-powered} (e.g. by \cite[Corollary 7.89]{ahs}), in the sense \cite[Definition 4.1.2]{brcx_hndbk-1} that the \emph{subobjects} of any given object (i.e. \cite[Definition 4.1.1]{brcx_hndbk-1} isomorphism classes of monomorphisms with a given common target) constitute a set. $\cC_{{\Bbbk}}$, on the other hand, are easily assessed to be non-well-powered: the objects $M^{\alpha}=\left(M^{\alpha}_{\beta}\right)_{\beta}$ defined on \cite[p.444]{MR4092830} as $0$ at $\beta<\alpha$ and $\Bbbk_{\beta}$ otherwise (with the initial inclusions $\Bbbk_{\beta}\le \Bbbk_{\beta'}$ as transition maps $\nu_{\beta'\beta}$) are all subobjects of $M^0$.
\end{remark}

In light of \Cref{re:no.gen.no.wp}, and of the crucial role played by both generators and well-powered-ness \cite[Lemma 3.10.5, Theorem 3.10.10]{pop_ab-cat} in ensuring the existence of \emph{injective envelopes} in Grothendieck categories, it seems apposite to examine these and ancillary notions (e.g. \emph{essential extensions}) with reference to the categories $\cC_{{\Bbbk}}$. 

\begin{lemma}\label{le:when.ess}
  A subobject $W:=(W_{\alpha})_{\alpha}\le (V_{\alpha})_{\alpha}=:V$ in $\cC_{{\Bbbk}}$ is embedded essentially only if $W_{\alpha}=V_{\alpha}$ for sufficiently large $\alpha\in \cat{ON}$. 
\end{lemma}
\begin{proof}
  Noting that the essential character of $W\le V$ ensures that of the truncated embedding 
  \begin{equation*}
    W_{\ge \alpha}
    \le
    V_{\ge \alpha}
    ,\quad
    \bullet_{\ge \alpha}
    :=
    \begin{cases}
      0&\text{under $\alpha$}\\
      \bullet&\text{at indices $\ge \alpha$}
    \end{cases}
  \end{equation*}
  for every $\alpha\in \cat{ON}$, we may assume $W$ and $V$ 0-grounded to begin with. Now, by the essential-embedding assumption again,
  \begin{equation*}
    \forall\alpha
    \forall\left(v\in V_{\alpha}\right)
    \exists\left(\beta>\alpha\right)
    \left(\nu_{\beta\alpha}v\in W_{\beta}\le V_{\beta}\right).
  \end{equation*}
  Fixing $\alpha$ and choosing $\beta>\alpha$ large enough for the latter to hold for all $v\in V_{\alpha}$, we have $\nu_{\beta\alpha}V_{\alpha}\le W_{\beta}$; this implies
  \begin{equation*}
    W_{\beta}=V_{\beta}
    \xRightarrow{\quad}
    \forall\left(\beta'\ge \beta\right)
    \left(W_{\beta'}=V_{\beta'}\right),
  \end{equation*}
  finishing the proof. 
\end{proof}

\Cref{re:no.gen.no.wp} notwithstanding, the (large) poset of subobjects of any given $V\in \cC_{{\Bbbk}}$ does satisfy the \emph{$\lambda$-ascending-chain condition} (so named after its familiar countable counterpart: \cite[Definition II.1.7]{kun_st_2nd_1983}, \cite[post Theorem 4.17.7]{pop_ab-cat}; shorthand: \emph{$\lambda$-ACC}) for sufficiently large cardinals $\lambda$: strictly ascending chains of subobjects have cardinality $<\lambda$. 

\begin{lemma}\label{le:lcc}
  If $V\in \cC_{{\Bbbk}}$ then the large poset of subobjects of $V$ is $\lambda$-ACC for sufficiently large cardinals $\lambda$.

  In fact, $\lambda$ can be chosen dependent on $\dim V_{\alpha'}$, $\alpha'\le \alpha$ as soon as $V$ is $\alpha$-grounded. 
\end{lemma}
\begin{proof}
  Assume $V$ $\alpha$-grounded. Given subobjects $W'\le W\le V$, we then have
  \begin{equation*}
    \forall\left(\beta>\alpha\right)
    \left(
      W_{\beta}/W'_{\beta}\ne 0
      \xRightarrow{\quad}
      \nu^{-1}_{\beta\alpha}\left(W_{\beta}/W'_{\beta}\right)
      \ne 0
    \right).
  \end{equation*}
  As any strict ascending chain of subspaces of $V_{\alpha}$ will of course be bounded by that space's dimension, any sufficiently long ascending chain of subobjects eventually stabilizes in cardinals $\ge \alpha$. This then forces an upper size bound dependent only on the dimensions of $V_{\alpha'}$, $\alpha'\le \alpha$.
\end{proof}

\begin{remark}\label{re:trnc.groth}
  In the process of constructing $\cC_{{\Bbbk}}$ \cite[Definition 2.2]{MR4092830} also introduces what we will here denote by $\cC_{{\Bbbk},\alpha}$: the full  subcategory consisting of the $\alpha$-grounded objects. \cite[Definition 2.2]{MR4092830} claims in passing that the $\cC_{{\Bbbk},\alpha}$ are all Grothendieck categories; I believe this cannot be the case: \Cref{le:trnc.chn.ess}. 
\end{remark}

\begin{lemma}\label{le:trnc.chn.ess}
  Fix ${\Bbbk}$ and an ordinal $\alpha$. Assuming $\cC_{{\Bbbk},\alpha}$ abelian, it admits proper-class-indexed chains of proper essential embeddings.

  In particular, $\cC_{{\Bbbk},\alpha}$ cannot be a Grothendieck category. 
\end{lemma}
\begin{proof}
  We focus on $\alpha:=0$, as the example already resides in that smallest of the $\cC_{{\Bbbk},\alpha}$. The sought-for chain consists of the objects
  \begin{equation*}
    X^{\alpha}
    =
    \left(X^{\alpha}_{\beta}\right)_{\beta}
    ,\quad
    X^{\alpha}_{\beta}
    :=
    \begin{cases}
      \Bbbk_{\alpha}&\beta\le \alpha\\
      \Bbbk_{\beta}&\beta>\alpha
    \end{cases}
  \end{equation*}
  with the obvious transition morphisms for each individual $X^{\alpha}$ (each $\Bbbk_{\beta}$ embeds in larger $\Bbbk_{\beta'}$ as in the initial datum ${\Bbbk}$) and the self-same transition morphisms also effecting the inclusions $X^{\alpha}\le X^{\alpha'}$ whenever $\alpha\le \alpha'$.
  
  As to the second claim: Grothendieck categories cannot house proper-class-long strictly increasing chains of essential extensions, as verified in the course of the proof of \cite[Theorem 6.25]{freyd_abcats}. 
\end{proof}

Consequently:

\begin{corollary}\label{cor:trnc.not.ab}
The truncated categories $\cC_{\Bbbk,\alpha}$ are not abelian. 
\end{corollary}
\begin{proof}
  That the $M^{\beta}$, $\beta\le \alpha$ of \Cref{re:no.gen.no.wp} constitute a set of generators is noted in \cite[Remark 2.6]{MR4092830}, and the AB5 condition (assuming abelianness) is essentially what the proof of \cite[Theorem 2.4]{MR4092830} argues. 
\end{proof}

The issue at the heart of \Cref{cor:trnc.not.ab} appears to also affect the incipient phase of the proof of \cite[Theorem 2.4]{MR4092830} (and is indeed easily checked; note however that the selfsame remark also claims that $\cC_{\Bbbk,\alpha}$ are Grothendieck): specifically, that proof claims early on without justification that (co)kernels of morphisms between $\alpha$-grounded objects are again such. This is unproblematic for \emph{co}kernels, as these are certainly compatible in any conceivable sense with the surjectivity that effectively powers grounded-ness, but kernels require some care.

\begin{example}\label{ex:ker.grnd}
  We will simply focus on objects of $\cC_{{\Bbbk},0}$ truncated over indices 0,1, as all such easily extend to full 0-grounded objects. Consider a commutative diagram
  \begin{equation*}
    \begin{tikzpicture}[>=stealth,auto,baseline=(current  bounding  box.center)]
      \path[anchor=base] 
      (0,0) node (l) {$V_0=W_0$}
      +(2,.5) node (u) {$V_1$}
      +(4,0) node (r) {$W_1$,}
      ;
      \draw[right hook->] (l) to[bend left=6] node[pos=.5,auto] {$\scriptstyle \text{$\Bbbk_0$-linear}$} (u);
      \draw[->>] (u) to[bend left=6] node[pos=.5,auto] {$\scriptstyle \text{$\Bbbk_1$-linear}$} (r);
      \draw[right hook->] (l) to[bend right=6] node[pos=.5,auto,swap] {$\scriptstyle \text{$\Bbbk_0$-linear}$} (r);
    \end{tikzpicture}
  \end{equation*}
  with $V_0$ generating $V_1$ over $\Bbbk_1$ (so that the left-hand up and bottom maps constitute truncated 0-grounded objects and the diagram is a morphism of such), and the right-hand upper map is a quotient by some $\Bbbk_1$-subspace meeting $V_0=W_0\le V_1$ trivially. Plainly, the kernel is \emph{not} 0-grounded. 
\end{example}

\Cref{ex:ker.grnd} raises some doubts over the validity of \cite[Theorem 2.4]{MR4092830} given that (as noted) the claim that kernels of $\alpha$-grounded objects are again such explicitly features in its proof. The abelianness of $\cC_{\Bbbk}$ itself, however, will follow once we relax that constraint to allowing kernels ``higher-degree grounded-ness''.

\begin{lemma}\label{le:hgh.grnd}
  The kernel of a morphism of grounded objects is again grounded. 
\end{lemma}
\begin{proof}
  Having chosen a $\cC_{\Bbbk}$-morphism $V\xrightarrow{f} W$, simply note that the $\alpha$-grounded truncations
  \begin{equation*}
    \downarrow_{\alpha}(\ker f)
    ,\quad
    \cC_{\Bbbk}
    \xrightarrow[\quad\substack{\text{right adjoint to inclusion}\\\text{\cite[pre Remark 2.6]{MR4092830}}}\quad]{\quad\downarrow_{\alpha}\quad}
    \cC_{\Bbbk,\alpha}
  \end{equation*}
  must stabilize by \Cref{le:lcc}.
\end{proof}

This in hand, the substance of the theorem appears to stand: $\cC_{\Bbbk}$ is abelian, despite its truncations' $\cC_{\Bbbk,\alpha}$ somewhat pathological (in retrospect) character.

\begin{theorem}\label{th:is.ab.no.env}
  The categories $\cC_{\Bbbk}$ are AB5 abelian, but do not have enough injectives. 
\end{theorem}
\begin{proof}
  Take abelianness for granted for the moment (thus confirming \cite[Theorem 2.4]{MR4092830} and hence also the AB5 property). The base $X^0$ of an ordinal-indexed essential chain $(X^{\alpha})_{\alpha}$ such as that provided by \Cref{le:trnc.chn.ess} cannot embed into an injective: transfinite induction would extend any embedding $X^0\lhook\joinrel\to I$ to $X^{\alpha}\lhook\joinrel\to I$ (again an embedding because $X^0\le X^{\alpha}$ is essential \cite[Lemmas 6.21, 6.22]{freyd_abcats}), whence a violation of \Cref{le:lcc} (which imposes eventual stabilization of any ordinal-indexed ascending chain of subobjects of $I$). 

  As to abelianness, the only missing ingredient is verifying that the kernel of a grounded morphism is again grounded; this being precisely what \Cref{le:hgh.grnd} ensures, the conclusion is confirmed. 
\end{proof}

\begin{remark}\label{re:other.ab5.no.inj}
  An AB5 category with no injective hulls (and indeed, no non-zero injective/projective objects) is also constructed in \cite[Exercise 6A]{freyd_abcats}. It \emph{is} well-powered (by contrast to $\cC_{\Bbbk}$), but of course still lacks a generator. It is also easily checked to be complete, so it will not do as a substitute for the original problem of retaining AB5 while avoiding AB3$^*$. 
\end{remark}

Unlike the categories mentioned in \Cref{re:other.ab5.no.inj}, $\cC_{\Bbbk}$ \emph{does} have non-zero injectives: those of certain Grothendieck subcategories of $\cC_{\Bbbk}$, as the sequel describes.

\begin{notation}\label{not:alpha.tors.thrs}
  For ordinals $\alpha$ set
  \begin{equation*}
    \begin{aligned}
      \cC_{\Bbbk}|_{\alpha}
      :=
      \left\{\left(V_{\beta}\right)_{\beta}\in \cC_{\Bbbk}\ :\ \forall\left(\beta\ge \alpha\right)\left(V_{\beta}=\{0\}\right)\right\}
      \overset{\text{full}}
      {\subseteq}
      \cC_{\Bbbk}
      ,\\
      \cC_{\Bbbk}|^{\alpha}
      :=
      \left\{\left(V_{\beta}\right)_{\beta}\in \cC_{\Bbbk}\ :\ \forall\left(\beta< \alpha\right)\left(V_{\beta}=\{0\}\right)\right\}
      \overset{\text{full}}
      {\subseteq}
      \cC_{\Bbbk};
    \end{aligned}    
  \end{equation*}
  the same conventions apply with $\ol{\cC}$, should the occasion arise. For obvious reasons, the objects of $\cC_{\Bbbk}|_{\bullet}$ will occasionally be referred to as \emph{eventually vanishing}.
\end{notation}

Some general remarks on the truncated categories just introduced follow. 

\begin{proposition}\label{pr:trnc.grth.tors}
  For every $\alpha\in \cat{ON}$:
  \begin{enumerate}[(1),wide]
  \item\label{item:pr:trnc.grth.tors:ex.groth} The subcategory $\cC_{\Bbbk}|_{\alpha}\subseteq \cC_{\Bbbk}$ is exact and Grothendieck.

  \item\label{item:pr:trnc.grth.tors:2ca} $\left(\cC_{\Bbbk}|^{\alpha},\cC_{\Bbbk}|_{\alpha}\right)$ is a \emph{hereditary torsion theory} \cite[\S 4.8]{pop_ab-cat} for $\cC_{\Bbbk}$.

  \item\label{item:pr:trnc.grth.tors:1ca} Dually, $\cC_{\Bbbk}|_{\alpha}$ are all hereditary torsion subcategories of $\cC_{\Bbbk}$, as is the subcategory $\cC_{\Bbbk}|_{\circ}:=\bigcup_{\alpha}\cC_{\Bbbk}|_{\alpha}$ of eventually vanishing objects. 
  \end{enumerate}
\end{proposition}
\begin{proof}
  To verify only \Cref{item:pr:trnc.grth.tors:1ca}, leaving the other, equally routine claims to the reader, simply apply the criteria of \cite[Theorem 4.8.4 and Corollary 4.8.5]{pop_ab-cat}:
  \begin{enumerate}[(a),wide]
  \item\label{item:pr:trnc.grth.tors:pf.her} that all $\cC_{\Bbbk}|_{\bullet}$ are hereditary (closed under taking subobjects) is immediate;
  \item\label{item:pr:trnc.grth.tors:pf.coher} they are all also \emph{co}hereditary (closed under quotients);

  \item\label{item:pr:trnc.grth.tors:sum.ext} and stable under both arbitrary sums and extensions. 
  \end{enumerate}
  Although well-powered-ness is assumed throughout \cite[\S 4.8]{pop_ab-cat}, an examination of the relevant proofs shows that all that is needed (given \Cref{item:pr:trnc.grth.tors:pf.her,item:pr:trnc.grth.tors:pf.coher,item:pr:trnc.grth.tors:sum.ext}) is the existence, for arbitrary objects in $\cC_{\Bbbk}$, of largest subobjects contained in $\cC_{\Bbbk}|_{\bullet}$. That, in turn, follows from \Cref{le:lcc}. 
\end{proof}

Recall \cite[\S 3.5]{brcx_hndbk-1} that \emph{(co)reflective} (full) subcategories are those whose inclusion functors have left (respectively right) adjoints. 

\begin{lemma}\label{le:trnc.co.refl}
  For $\alpha\in\cat{ON}$, the subcategories $\cC_{\Bbbk}|_{\alpha}$ and $\cC_{\Bbbk}|^{\alpha}$ are respectively reflective and coreflective in $\cC_{\Bbbk}$.
\end{lemma}
\begin{proof}
  That the truncation functors
  \begin{equation}\label{eq:trncs}
    \begin{aligned}
      \cC_{\Bbbk}
      \ni
      \left(V_{\beta}\right)_{\beta}
      \xmapsto{\quad\bullet|_{\alpha}\quad}
      \left(
      \begin{aligned}
        V_{\beta},\quad&\beta<\alpha\\
        \{0\},\quad&\beta\ge \alpha
      \end{aligned}
      \right)\in \cC_{\Bbbk}|_{\alpha}\\
      \cC_{\Bbbk}
      \ni
      \left(V_{\beta}\right)_{\beta}
      \xmapsto{\quad\bullet|^{\alpha}\quad}
      \left(
      \begin{aligned}
        \{0\},\quad&\beta<\alpha\\
        V_{\beta},\quad&\beta\ge \alpha
      \end{aligned}
      \right)\in \cC_{\Bbbk}|^{\alpha}
    \end{aligned}    
  \end{equation}
  are respectively left and right adjoint to the two inclusions is immediate. 
\end{proof}

As to injective objects in $\cC_{\Bbbk}$, \Cref{th:inj.event.vnsh} confirms that the Grothendieck subcategories $\cC_{\Bbbk}|_{\alpha}$ of eventually vanishing objects account for all of them. 

\begin{theorem}\label{th:inj.event.vnsh}
  The injectives of $\cC_{\Bbbk}$ are precisely those objects that are injective in some Grothendieck subcategory $\cC_{\Bbbk}|_{\alpha}$ of eventually vanishing objects. 
\end{theorem}
\begin{proof}
  Two claims constitute the result.
  
  \begin{enumerate}[(I),wide]
  \item\textbf{: Injectives in $\cC_{\Bbbk}$ eventually vanish.} The construction proving \Cref{le:trnc.chn.ess} of course functions in arbitrary $\cC_{\Bbbk}|^{\alpha}$, so the object $T^{\alpha}\in \cC_{\Bbbk}|^{\alpha}$ whose $\beta$-component is $\Bbbk_{\beta}$ for $\beta\ge \alpha$ (and $\{0\}$ otherwise) is the basis of a proper-class-long chain of essential extensions and cannot embed in an injective. To conclude, it remains to observe that 
    \begin{equation*}
      V\in \cC_{\Bbbk}\text{ does not eventually vanish}
      \iff
      \exists\left(\alpha\in\cat{ON}\right)
      \exists\left(T^{\alpha}\lhook\joinrel\xrightarrow{\quad} V\right).
    \end{equation*}
    Indeed, there is an embedding $T^{\alpha}\lhook\joinrel\to V$ as soon as $V$ is $\alpha$-grounded: it suffices to argue that some $v\in V_{\alpha}$ does not eventually vanish (i.e. $\nu_{\beta\alpha}v\ne 0$ for all $\beta>\alpha$). This in turn follows from the observation that the subspace
    \begin{equation*}
      \left\{v\in V_{\alpha}\ :\ \exists\left(\beta\in\cat{ON}\right)\left(\nu_{\beta\alpha}v=0\right)\right\}
      \le
      V_{\alpha}
    \end{equation*}
    of eventually vanishing elements in $V_{\alpha}$ will be annihilated by $\nu_{\beta\alpha}$ for a single sufficiently large $\beta$, whence the eventual vanishing of $V$ as a whole by $\alpha$-grounded-ness.
    
  \item\textbf{: Injectives in $\cC_{\Bbbk}|_{\alpha}$ remain injective in $\cC_{\Bbbk}$.} The inclusion functor of the smaller category in the larger is right adjoint to the upper exact functor \Cref{eq:trncs}, so preserves injectivity. 
  \end{enumerate}
\end{proof}

We also record a stronger form of stability for grounded objects that can be bootstrapped out of the initial defining requirement that $\Bbbk_{\beta}\otimes_{\Bbbk_{\alpha}}V_\alpha\to V_{\beta}$ being onto. 

\begin{definition}\label{def:atrm}
  We refer to an object $V=(V_{\alpha})_{\alpha}$ as \emph{$\alpha$-trim} if $\nu_{\gamma\beta}$ are injective for $\gamma\ge \beta\ge \alpha$. An object is \emph{trim} if $\alpha$-trim for some ordinal $\alpha$.
\end{definition}

\begin{lemma}\label{le:ev.inj}
  Every $V=(V_{\alpha})_{\alpha}\in \cC_{\Bbbk}$ is trim.

  Equivalently, the induced maps $\Bbbk_{\beta}\otimes_{\Bbbk_{\alpha}}V_\alpha\to V_{\beta}$ are isomorphisms for $\beta\ge \alpha\gg 0$. 
\end{lemma}
\begin{proof}
  That the two claims are content-wise equivalent is self-evident. Moreover, there is no loss of generality in assuming $V$ 0-grounded (for we can always focus only on ordinals dominating the grounded-ness level).

  Assuming $\ker \nu_{\beta\alpha}\ne \{0\}$ for arbitrarily large $\beta\ge \alpha$,
  \begin{equation*}
    \exists\left(\alpha\in\cat{ON}\right)
    \left(
      \sharp\left\{\beta\in\cat{ON}\ :\ \beta< \alpha,\ \ker\nu_{\beta+1,\beta}\ne \{0\}\right\}
      >
      \dim V_0
    \right).
  \end{equation*}
  Writing $\Bbbk_{\alpha}\otimes_{\Bbbk_0}V_0\xrightarrowdbl{\pi}V_{\alpha}$, the subspaces $\Bbbk_{\alpha}\pi^{-1}V_{\beta}$ constitute an ascending chain in $\Bbbk_{\alpha}\otimes_{\Bbbk_0}V_0$, strictly so over an index set of cardinality larger than that space's dimension. Plainly, this is a contradiction. 
\end{proof}

%%%%%%%%%%%%%%%%%%%%%%%%%%%%%%%%
%%%%%%%%%%%%%%%%%%%%%%%%%%%%%%%%
\section{Set-decorated abelian categories}\label{se:set.dec}

In light of \Cref{th:inj.event.vnsh}, it will not be unnatural to investigate whether the $\cC_{\Bbbk}$ examples can in some fashion be conjoined with the completely injective-deficient (but complete) category ($\cA$, say) mentioned in \Cref{re:other.ab5.no.inj} so as to produce incomplete AB5 categories with no (non-zero) injectives whatsoever. This is indeed possible by abstracting away from $\cA$ (which we recall explicitly in the process) the relevant machinery producing new abelian categories out of old.

\begin{definition}\label{def:dec}
  For an \emph{additive category} $\cC$ (i.e. \cite[\S VIII.2]{mcl_2e} one enriched over the category $\cat{Ab}$ of abelian groups and admitting finite (co)products) the category $\cC\triangleleft\cat{Set}$ of \emph{set-decorated $\cC$-objects} (or: the \emph{set decoration of $\cC$}) has
  \begin{itemize}[wide]
  \item triples
    \begin{equation*}
      (c,S,f)
      \quad:\quad
      c\in \cC
      ,\quad
      S\in \cat{Set}
      ,\quad
      S\xrightarrow[\quad \text{function}\quad]{f}\cC(c,c)
    \end{equation*}
    as objects;

  \item with morphisms $(c,S,f)\xrightarrow{\psi}(c',S',f')$ defined as those $\cC$-morphisms $\psi\in \cC(c,c')$ making
    \begin{equation*}
      \begin{tikzpicture}[>=stealth,auto,baseline=(current  bounding  box.center)]
        \path[anchor=base] 
        (0,0) node (l) {$c$}
        +(2,.5) node (u) {$c$}
        +(2,-.5) node (d) {$c'$}
        +(4,0) node (r) {$c'$}
        ;
        \draw[->] (l) to[bend left=6] node[pos=.5,auto] {$\scriptstyle f(s)$} (u);
        \draw[->] (u) to[bend left=6] node[pos=.5,auto] {$\scriptstyle \psi$} (r);
        \draw[->] (l) to[bend right=6] node[pos=.5,auto,swap] {$\scriptstyle \psi$} (d);
        \draw[->] (d) to[bend right=6] node[pos=.5,auto,swap] {$\scriptstyle f'(s)$} (r);
      \end{tikzpicture}
    \end{equation*}    
    commutative for all $s\in S\cup S'$;

  \item where evaluation of the $f$ component of $(c,S,f)$ on elements outside $S$ yielding $0\in \cC(c,c)$.
  \end{itemize}
  $\bullet\triangleleft \cat{Set}$ is what in 2-categorical parlance is referred to as a \emph{pseudofunctor} \cite[Definition 4.1.2]{jy_2dcat}. 
\end{definition}

The category referenced in \Cref{re:other.ab5.no.inj}, then, is $\cat{Ab}\triangleleft \cat{Set}$. A simple remark:

\begin{lemma}\label{le:prsv.cmpl}
  \begin{enumerate}[(1),wide]
  \item\label{item:le:prsv.cmpl:colims} The pseudofunctor $\bullet\triangleleft\cat{Set}$ preserves abelianness as well as $\cD$-shaped (co)limits for any small category $\cD$, and when those exist in $\cC$ the forgetful functor $\cC\triangleleft \cat{Set}\xrightarrow{U=U_{\cC}} \cC$ preserves them. 

  \item\label{item:le:prsv.cmpl:ab} $\bullet\triangleleft \cat{Set}$ also preserves the AB$\bullet$ conditions for $\bullet\in \left\{3,3^*,4,4^*,5,5^*\right\}$.

  \item\label{item:le:prsv.cmpl:pow} The forgetful functor $U$ preserves and reflects (co-)-well-powered-ness, as well as the $\lambda$-ACC condition of \Cref{le:lcc} and its dual descending counterpart, on posets of either subobjects or quotient objects. 
  \end{enumerate}
\end{lemma}
\begin{proof}
  \begin{enumerate}[(1),wide]
  \item Focusing on products only to fix ideas (as the rest is not materially more difficult), simply note that
    \begin{equation*}
      c:=\prod_{i\in I}c_i
      ,\quad
      S:=\bigcup_i S_i
      ,\quad
      S\ni s
      \xmapsto{\quad f\quad}
      \prod_i f_i(s)\in \cC(c,c)
    \end{equation*}
    is a product of $(c_i,S_i,f_i)$ in $\cC\triangleleft\cat{Set}$, recalling the convention that $f_i(s)$ defaults to 0 when $s\not\in S_i$.

  \item For AB5, say: $U$ reflects monomorphisms, so that
    \begin{equation*}
      \sum_i\left((d,T,g)\cap (c_i,S_i,f_i)\right)
      =
      (d,T,g)\cap \sum_i(c_i,S_i,f_i)
    \end{equation*}
    for subobjects $(d,T,g),(c_i,S_i,f_i)\le (c,S,f)$ in $\cC\triangleleft\cat{Set}$ follows from the analogous equality for $d,c_i\le c$ in $\cC$. The other claims are, if anything, simpler.

  \item The claims all follow from the remark that for a subobject $(c',S',f')\le (c,S,f)$, $f'(s')$ must vanish whenever $f(s')$ does (so in particular, whenever $s'\not\in S$). Consequently, for a class
    \begin{equation*}
      \cS=\left\{(c_s,S_s,f_s)\le (c,S,f)\right\}
    \end{equation*}
    of subobjects, the corresponding class
    \begin{equation*}
      \left\{c'\le c\ :\ \exists\left(s\in \cS\right)\left(\left(c'\le c\right)\cong \left(c_s\le c\right)\right)\right\}
    \end{equation*}
    is a set if and only if $\cS$ was one to begin with. 
  \end{enumerate}
\end{proof}

\begin{remark}\label{re:saft}
  Recall the automatic (left) adjointness \cite[\S V.8, Theorem 2]{mcl_2e} of cocontinuous functors defined on cocomplete, co-well-powered categories with generating sets. As a manifestation of generator absence, \cite[Foreword, p.-16]{freyd_abcats} that the forgetful functor $\cat{Ab}\triangleleft\cat{Set}\to \cat{Ab}$ illustrating adjunction failure: it is cocontinuous, with no adjoint of either type.

  Indeed: right (left) adjoints of exact functors preserve injectivity (projectivity) \spr{015Z}, so either adjoint companion's existence would contradict $\cC$'s having no non-zero injectives or projectives (which property \cite[Exercise 6A(5)]{freyd_abcats} claims).
\end{remark}

\begin{theorem}\label{th:incmpl.no.inj}
  For a $\cat{ON}$-indexed tower \Cref{eq:kappa} of infinite field extensions the category $\cC_{\Bbbk}\triangleleft\cat{Set}$ is AB5, does not have $\aleph_0$-indexed products, and has no non-zero injectives or projectives. 
\end{theorem}
\begin{proof}
  The (general) absence of countable products follows from its counterpart (\cite[proof of Theorem 2.4]{MR4092830}) in $\cC_{\Bbbk}$: observe that $\bullet\triangleleft \cat{Set}$ also \emph{reflects} the existence of $I$-indexed products (in addition to preserving them, conversely, when they are known to exist in the original category), since the underlying $\cC$-object $c$ of a product
  \begin{equation*}
    (c,S,f)
    :=
    \prod_i (c_i,\emptyset,\emptyset)
  \end{equation*}
  must be a product $\prod_i c_i$ in $\cC$. Abelianness and AB5 follow from \Cref{le:prsv.cmpl}\Cref{item:le:prsv.cmpl:ab}. As to the no-injectives/projectives claim, it follows from the fact (\Cref{le:c.in.cc}) that arbitrary $(V,S,f)\in \cC_{\Bbbk}\triangleleft \cat{Set}$ with non-zero underlying $V\in \cC_{\Bbbk}$ have proper essential extensions in both senses (i.e. are both domains of essential embeddings and codomains of essential epimorphisms). 
\end{proof}

\begin{lemma}\label{le:c.in.cc}
  For abelian $\cC$ an arbitrary object $(c,S,f)\in \cC\triangleleft\cat{Set}$ is the domain of an essential monic and the codomain of an essential epic, with the other object of the form $(c\oplus c,S',f')$.

  In particular, non-zero objects have proper essential extensions both as subobjects and quotient objects. 
\end{lemma}
\begin{proof}
  We will focus on monomorphisms, recycling the device employed in \cite[Exercise 6A]{freyd_abcats} in proving that $(\bZ,\emptyset,\emptyset)$ cannot embed into an injective:
  \begin{equation*}
    S':=S\sqcup \{s_0\}
    \quad
    \left(\text{disjoint union}\right)
    ,\quad
    f'|_S=f
    ,\quad
    f'(s_0)
    :=
    \begin{pmatrix}
      0&\id\\
      0&0
    \end{pmatrix}
    \in
    \cC(c\oplus c,c\oplus c).
  \end{equation*}
  The desired essential extension $(c,S,f)\le (c\oplus c,S',f')$ identifies $c$ with the \emph{first} summand (whereas $f'(s_0)$ maps the \emph{second} onto the first identically). 
\end{proof}

%%%%%%%%%%%%%%%%%%%%%%%%%%%%%%%%
%%%%%%%%%%%%%%%%%%%%%%%%%%%%%%%%

\addcontentsline{toc}{section}{References}
%\bibliography{bib}{}
%\bibliographystyle{plain}

% BEGIN INSERTED BBL (lrg-ab5-xv1.bbl)
\def\polhk#1{\setbox0=\hbox{#1}{\ooalign{\hidewidth
  \lower1.5ex\hbox{`}\hidewidth\crcr\unhbox0}}}

% END INSERTED BBL

\Addresses

\end{document}